%% file: main.tex
\documentclass[letterpaper, 10 pt, conference]{ieeeconf}  
\usepackage{style}
\usepackage[english]{babel}
\IEEEoverridecommandlockouts                              
\title{\LARGE \bf
Efficient Rank Minimization to Tighten Semidefinite Programming for Unconstrained Binary Quadratic Optimization
}

\author{ Roman Pogodin$^{1}$, Mikhail Krechetov$^{2}$ and Yury Maximov$^{3}$
\thanks{The work done at Skolkovo Institute of Science and Technology is supported partly by the grant of the President of Russian Federation for young PhD MK-9662.2016.9 and partly by the RFBR grant 15-07-09121a. The work at LANL was carried out under the auspices of the National Nuclear Security Administration of the US Department of Energy under Contract No. DE-AC52-06NA25396.}
\thanks{$^{1}$Center for Energy Systems, Skolkovo Institute of Science and Technology
	Skolkovo Innovation Center, Building 3, 
	Moscow 143026 Russia
  \texttt{R.Pogodin@skoltech.ru}}
\thanks{$^{2}$Center for Energy Systems, Skolkovo Institute of Science and Technology
	Skolkovo Innovation Center, Building 3, 
	Moscow 143026 Russia
  \texttt{M.Krechetov@skoltech.ru}}%
\thanks{$^{3}$Center for Energy Systems, Skolkovo Institute of Science and Technology, and
  Theoretical Division T-4 and CNLS, Los Alamos National Laboratory 
  MS-B258,  Los Alamos, 87545 NM, USA
  \texttt{yury@lanl.gov}}%
}

\begin{document}

\maketitle
\thispagestyle{empty}
\pagestyle{empty}

\begin{abstract}

We propose a method for low-rank semidefinite programming in application to the semidefinite relaxation of unconstrained binary quadratic problems. The method improves an existing  solution of the semidefinite programming relaxation to achieve a lower rank solution. This procedure is computationally efficient as it does not require projecting on the cone of positive-semidefinite matrices. Its performance in terms of objective improvement and rank reduction is tested over multiple graphs of large-scale Gset graph collection and over binary optimization problems from the Biq Mac collection.

\end{abstract}

\section{Introduction}

Binary quadratic optimization is a classical combinatorial optimization problem, which finds a wide range of applications in computer vision \cite{wang2017large, ren2014optimizing}, circuit layout design \cite{Barahona,kleinhans1991gordian},   computing ground states of Ising model \cite{MCising}, as well as a number of combinatorial favors \cite{parker2014discrete,garey2002computers,cormen2009introduction}.  For  comprehensive  list of applications we refer to \cite{Deza, UBQP}.

A special case of this problem is unconstrained binary quadratic programming (UBQP). It is a classical  NP-hard problem, hardly possible to be solved exactly in polynomial time. A number of relaxation techniques substituting the original problem to a convex one has been proposed. Linear, second-order cone and semidefinite relaxations are among them. 

Semidefinite programming relaxation (SDP) has been shown to lead to tighter approximation than other relaxation methods for many combinatorial optimization problems including  binary quadratic optimization ones (\cite{Goemans:1995:IAA:227683.227684} and \cite{Khot:2002:PUG:509907.510017}).  Still, SDP reduces the problem to convex optimization over the cone of positive semidefinite matrices and outputs a full-rank matrix requiring to be rounded  to obtain a vector valued solution. 

In this paper we address the question of low-rank semidefinite programming, which is aimed at strengthening results of the standard SDP relaxation. While a number of methods has been proposed (see \cite{Lemon2016} for a survey),  we discuss direct rank minimization of an obtained SDP solution. We use smoothed rank approximations in order to reduce the rank of the SDP solution without significant loss in the optimal value. For this purpose we propose an efficient first-order optimization procedure which does not require projecting on the feasible set. This will potentially lead to more accurate rounding procedure allowing to obtain a better vector solution.

In the rest of the paper we discuss problems of the form
\begin{equation}
\begin{split}
\max\ &x\trans Ax,\\
\st\ &x\in\{-1, 1\}^n,
\end{split} 
\label{eq::bin_quadr_problem}
\end{equation}
where $A$ is an arbitrary symmetric $n\times n$ matrix.

There are many well-known problems that can be naturally written in this form: the maximum cut problem, the 0-1 knapsack problem, the linear quadratic regulator and many others. 

Not all of these formulations appear in the form \ref{eq::bin_quadr_problem}. Instead, some problems might have a linear term of the from $b\trans x$ and a 0-1 constraint, that is $x\in \{0,1\}^n$. Nevertheless, such problems might be converted to the form \ref{eq::bin_quadr_problem}, as showed in \cite{Helmberg1995}. The corresponding derivation is showed in the appendix.

For instance, in the maximum cut (max-cut) problem we want to find a partition of graph's vertices into two disjoint sets such that the sum of edges between these sets is maximal. This problem is NP-hard \cite{Karp1972}, however, it can be solved approximately. For small instances (up to $50$ vertices) the maximum cut may be solved efficiently with the branch-and-bound algorithm \cite{Malick}, \cite{Rinaldi}. For bigger problems (around $1000$ vertices) the semidefinite relaxation discussed below gives the best known approximations. It is crucial for many applications to improve existing approximations or to extend them on large-scale instances.

Quadratic boolean programming (\ref{eq::bin_quadr_problem}) is a particular case of quadratically constrained quadratic problems (QCQP), so general heuristics for this class of problems may be applied. See for example the recent paper \cite{ParkBoyd}.

\subsection{Semidefinite Relaxation.}
Standard semidefinite (SDP) relaxation leads to the following matrix problem:
\begin{equation}
\begin{split}
\max\ & \tr(AX),\\ 
\st\ &\diag X = 1_n,\\
&X\trans=X,\ X\succeq 0. 
\end{split}
\label{eq::sdp}
\end{equation}
This problem is convex and thus could be efficiently solved (we will discuss the particular method below). 

To get a binary solution of the initial problem \ref{eq::bin_quadr_problem}, we decompose the solution of the SDP relaxation $X=V\trans V$ (via Cholesky decomposition), then take a unit vector with uniformly distributed direction $r$. For each column of $V$, which is $v_i$, we take $x_i=\sign v_i\trans r$. If $A\succeq 0$, the mean result of this procedure is not worse than $2/\pi$ of the maximum value \cite{NesterovSDP2}. In a special case when $A$ is a Laplacian of a graph with non-negative weights this bound can be further improved to $\approx 0.878$ \cite{Goemans:1995:IAA:227683.227684}. Note that this famous results cannot be improved if the Unique Games Conjecture is true \cite{Khot:2002:PUG:509907.510017}.

However, this relaxation is exact, when $\rank X=1$. Moreover, low-rank solutions lead to a fewer number of possible binary sets in the rounding procedure described above. This idea becomes more clear if one considers half-space classifiers in $\RR^r$ and $n$ points. The maximum number of different labellings is controlled by Sauer's lemma, and grows as $(n+1)^r$. A more accurate discussion of Sauer's lemma can be found in \cite{Bartlett2003}. Hence, if we lie in the vicinity of a correct solution, we would get the correct result from the rounding procedure more likely. This motivation brings us to the idea of low-rank semidefinite programming.

\subsection{Related Work.}
The existence of low-rank solutions of the problem \ref{eq::sdp} is a fundamental fact discussed in \cite{Pataki98} and \cite{Barvinok1995}. From these works we know that for such problems there exist solutions of the rank at most $r$, where $r(r+1)\leq 2n$. 

Knowing about the existence of low-rank solutions, we may discuss popular approaches in low-rank semidefinite programming. This section is mostly based on the book \cite{Lemon2016}.

Firstly, the existence of low-rank solutions is used in Burer and Monteiro method \cite{Burer2003}. It is based on the factorization $X=VV\trans$, where $V$ is an arbitrary matrix of size $n\times r$. This problem is non-convex, though it requires much less computations and performs well in practice. However, finding the minimum rank solution of multiple runs of the algorithm with different $r$ (one increments $r$ until resulting point satisfies particular conditions) might be computationally ineffective.

Another approach implies relaxation of the equality constraint \cite{So2008}. It can be written as following:
\[
\tr \brackets{AX}=b \Rightarrow \beta b\leq \tr\brackets{AX}\leq \alpha b,
\] 
where $\alpha$ and $\beta$ control the rank of the solution. In our case the problem \ref{eq::sdp} has $n$ equality constraints of the form $\tr\brackets{XE_{ii}}=1$, where $E_{ii}$ is a zero matrix with unit in the position $i,i$. All $n$ constraints together form $\diag X=1$. Though this approach allows to reduce the rank, the resulting solution satisfies our constraints only approximately. 

The next approach implies that we have already got a solution of the problem \ref{eq::sdp}. This allows us to reduce its rank via some kind of rank minimization procedure keeping the value of $\tr\brackets{AX}$ optimal. Note that rank minimization is NP-hard, so this formulation needs to be further relaxed to be efficiently solvable. This approach has been discussed in the literature, but we have not found a comparison of different ways to minimize the rank in application to boolean quadratic programming.

In our work we propose an efficient first-order algorithm, which performs rank minimization. It starts from the solution of the SDP relaxation. This solution is provided by either CVX interior-point algorithm \cite{cvx}, \cite{gb08} or Burer-Monteiro low-rank procedure, implemented in SDPLR \cite{Burer2003}.

\section{Rank minimization}

In this section we introduce the problem of rank minimization for the SDP relaxation. After that we describe existing approaches for rank minimization, and discuss their pros and cons.

\subsection{Problem.}

Let $X^*$ be a solution of \ref{eq::sdp} and let $SDP$ be its value and $W^*$ be the value of the binary solution obtained by the rounding procedure. Starting from $X^*$, we want to solve the following problem:
\begin{equation}
\begin{split}
\min\ &\rank X\\
\st\ &\diag X = 1_n,\\
&X\trans=X,\ X \succeq 0,\\
&\tr (AX)= SDP,
\end{split}
\label{eq::min_rank_norelax}
\end{equation}

\subsection{Objective function's relaxations.} 
However, minimizing the rank is NP-hard. In order to solve the problem \ref{eq::min_rank_norelax}, we replace $\rank X$ with a smooth surrogate, usually non-convex. There is a number of rather popular ways to do that, discussed below.

First of all, the so-called trace norm (or nuclear), defined as 
\begin{equation}
\|X\|_*=\sum_i\sigma_i,
\label{eq::trace_norm}
\end{equation}
where $\sigma_i$ is the $i$-th singular value \cite{Lemon2016}. In our case, every feasible point of the problem has $\tr X=\|X\|_*=n$, hence this relaxation does not make sense.

Next, the so-called log-det heuristic is to replace $\rank X$ with the concave function
\begin{equation}
\log\det\brackets{X+\eps I}.
\label{eq::logdet}
\end{equation}
This heuristic is discussed, for example, in \cite{Lemon2016}, \cite{Fazel2003}. Though it performs well in practice, it requires an iterative procedure (described in \cite{Fazel2003}) with an SDP problem on each iteration. This problem allows to use Burer-Monteiro method \cite{Burer2003}, but it still needs several runs of this algorithm (at least one for each iteration), which is compatible with rank increment in the original Burer-Monteiro procedure. 

The next two relaxations are singular value-based, and in the next section we show that, in fact, they allow an efficient first-order procedure, that does not require projections on the semidefinite cone and hence is computationally efficient.

The first one is the non-convex function of the following form
\begin{equation}
\Phi (X,\,\eps )=(1+\eps^q)\tr\brackets{X\trans(XX\trans+\eps I)^{-1}X}.
\label{eq::rank_singular}
\end{equation}
for $q\in\mathbb{Q}\cap\left[0,1\right]$. Its properties are discussed in \cite{Li}. An important fact is that this relaxation is quite close to the rank:
\[
\begin{split}
&\left|\rank X - \Phi(X,\,\eps) \right|\leq \\
&\leq\eps^q\max\left\lbrace\rank X, \sum_{i=1}^{\rank X}\left|\frac{\eps^{1-q}}{\sigma_i^2(X)}-1\right|\right\rbrace.
\end{split}
\]

The second relaxation utilizes so-called smoothed Schatten p-norms. They are defined as following:
\begin{equation}
\schatteneps{X}{p}^p = \sum\limits_{i\geq 1} \left(\sigma_i^2 + \eps\right)^{\frac{p}{2}} = \tr\left(X\trans X+\varepsilon I\right)^{\frac{p}{2}}.
\label{eq::schatten_norm_smoothed}
\end{equation}
With $p\rightarrow 0$ and $\eps=0$ we get the rank function exactly. Note that for $p<1$ this function is non-convex, and for $p=1$ it is identical to the nuclear norm. Applications of this relaxation can be found in \cite{Nie:2012} and \cite{Mohan2012}. Both papers introduce an iteratively reweighted least squares (IRLS) algorithm in order to solve this problem. However, in our case it requires solving of an SDP problem with quadratic objective function at each iteration. Hence, it cannot be applied to large-scale problems.

\subsection{Constraint relaxation.}
If we omit the last constraint in the problem \ref{eq::min_rank_norelax}, which is $\tr\brackets{AX}=SDP$, rank minimization might occasionally converge to a solution of rank one. Moreover, it may converge to a low-rank vicinity of such solution. If the SDP relaxation is not tight, then this constraint prevents the procedure from such behavior. 

We can also obtain a binary solution after solving the SDP relaxation. If we denote the objective value at this point $W^*$, then this value would be a natural lower bound on $\tr\brackets{AX}$. It means that all solutions of rank one above this value are actually better, than the one we got.

This motivation allows us to relax the problem further, and solve (along with rank approximations) the following one:
\begin{equation}
\begin{split}
\min\ &\rank X\\
\st\ &\diag X = 1_n,\\
&X\trans=X,\ X \succeq 0,\\
W^*\le\,&\tr (AX)\le SDP.
\end{split}
\label{eq::min_rank}
\end{equation}
Obviously, the binary solution, that gives $W^*$, is also a solution of the last problem. However, the typical procedure for solving \ref{eq::min_rank} would start from the SDP matrix in order to improve the resulting cut. Since the rank of this matrix is not unit in general, we need to optimize the objective function further.  

An important consequence of this relaxation will be clear in the next section as it allows to avoid projecting on the feasible set.

For completeness we emphasize that our approach cannot be generalized to QCQP problems. To avoid projecting on the set, it relies significantly on the special structure of constraints that occur in the problem \ref{eq::min_rank}.

\section{Main result}
In this section we introduce an efficient first-order procedure that solves the problem \ref{eq::min_rank} without projecting on the positive-semidefinite cone. It can be applied to the singular value relaxation \ref{eq::rank_singular} and smoothed Schatten p-norm \ref{eq::schatten_norm_smoothed}.

\subsection{Efficient first-order procedure.}
The problem \ref{eq::min_rank} allows a natural reparametrization to a vector problem of dimension $n(n-1)/2$. To do that, we consider the upper triangular part of the matrix:
\begin{equation}
X=\begin{pmatrix}
1 & x_1 & \dots & x_{n-1}\\
x_1 & 1 & x_n & \dots\\
\dots & \dots & \dots &\dots
\end{pmatrix}.
\label{eq::matrix_vectorized}
\end{equation}
In this case for indices $i,j$ we get $X_{i,j}=x_d$, where $d=\sum_{k=1}^{i-1}(n-k)+j-i$. This satisfies two constraints immediately: $X\trans=X$ and $\diag X=1_n$.

Such reparametrization changes the gradient of the matrix function $f(X)$:
\[
\begin{split}
\parderiv{f(X(x))}{x_{(d)}}=\sum_{i,j}\parderiv{f(X)}{x_{i,j}}\parderiv{x_{i,j}}{x_{(d)}}=\parderiv{f(X)}{x_{i',j'}}\parderiv{x_{i',j'}}{x_{(d)}}+\\
+\parderiv{f(X)}{x_{j',i'}}\parderiv{x_{j',i'}}{x_{(d)}}=\parderiv{f(X)}{x_{i',j'}}+\parderiv{f(X)}{x_{j', i'}},
\end{split}
\]
where $i', j'$ relate to the vector of index $d$. 

Obviously, the upper bound $\tr (AX)\le SDP$ is always satisfied. Moreover, violation of the lower one $W^*\le\tr (AX)$ implies that we got the point that could not improve our binary solution, hence we need to stop.

The last constraint is $X\succeq 0$. We show that proper choice of the gradient step results in a feasible point in case of singular value relaxation and Schatten norms.

First of all, the gradient of the singular value relaxation is
\begin{equation}
\parderiv{\Phi(X,\,\eps)}{X}=2\eps(1+\eps^q)\brackets{XX\trans+\eps I}^{-2}X.
\label{eq::rank_singular_two_deriv}
\end{equation}
For Schatten p-norms we have
\begin{equation}
\derivative{\schatteneps{X}{p}^p}{X}=pX\left(X\trans X+\varepsilon I\right)^{\frac{p-2}{2}}.
\label{eq::schatten_smoothed_derivative}
\end{equation}

If $X$ is symmetric and PSD, then from SVD factorization both gradients are symmetric. Thus in vector parametrization we simply need to multiply the gradient by $2$, and then force diagonal elements to be unit.

For further convenience we denote a symmetric matrix with unit diagonal, upper triangular part of which is constructed from the vector $x$, as $X(x)$. Finally, we show the following:

\begin{theorem}
Let $f(x)$ be the vector-parametrized singular value relaxation \ref{eq::rank_singular} of the matrix $X(x)\succeq 0$. Then for $\alpha\leq \frac{\eps}{4(1+\eps^q)}$ we get $X(x-\alpha \nabla f(x))\succeq 0$.
\label{theorem::singval_psd_constraint}
\end{theorem}
\begin{proof}
The gradient step in upper-triangular parametrization is equivalent to the ordinary gradient step (multiplied by $2$), and then substituting diagonal elements to units. We are going to show that the first step results in a PSD matrix and then the second step keeps matrix PSD.

Consider a symmetric PSD point and its SVD decomposition $X=USU\trans$. Hence for a step $\alpha$ and $C=2\eps (1+\eps^q)$ the new point is
\[
\begin{split}
&X-2\alpha C(XX\trans+\eps I)^{-2}X=\\
&=USU\trans-2\alpha C\brackets{US^2U\trans+\eps I}^{-2}USU\trans=\\
&=U\brackets{S-2\alpha C\brackets{S^2+\eps I}^{-2}S}U\trans.
\end{split}
\]

Hence the positive-semidefiniteness of the resulting matrix is equivalent to such characteristic of the expression in brackets. It is a diagonal matrix. 

If $S_{ii}=0$, then the corresponding diagonal elements are obviously zero. Otherwise we need it to be positive:
\[
\begin{split}
&S_{ii}-2\alpha C\brackets{S_{ii}^2+\eps I}^{-2}S_{ii}\geq 0\Rightarrow\\
&\Rightarrow \alpha \leq \frac{\brackets{S_{ii}^2+\eps I}^{2}}{2 C}=\frac{\brackets{S_{ii}^2+\eps I}^{2}}{4\eps(1+\eps^q)}.
\end{split}
\]
Therefore, it is enough to take
\[
\alpha \leq \frac{\eps}{4(1+\eps^q)}.
\]

Now we want to show that substituting diagonal elements with units is equivalent to adding a diagonal matrix with non-negative entries. In this case, a sum of two PSD matrices is PSD.

It is also equivalent to the fact that all diagonal elements of the gradient are non-negative. This is always true for a symmetric and PSD matrix $X=USU\trans$:
\[
\begin{split}
&\brackets{(XX\trans+\eps I)^{-2}X}_{ii}=\brackets{U\brackets{S^2+\eps I}^{-2}SU\trans}_{ii}=\\
&=\sum_{j}U_{ij}U_{ij}\brackets{S_{jj}^2+\eps I}^{-2}S_{jj}=\\
&=\sum_{j}U_{ij}^2\brackets{S_{jj}^2+\eps I}^{-2}S_{jj}\geq 0.
\end{split}
\]
This observation completes the proof.
\end{proof}

The same technique allows us to get a similar result for smoothed Schatten p-norms:
\begin{restatable}{theorem}{schattentheorem}
Let $f(x)$ be vector-parametrized smoothed Schatten p-norm of a matrix $X(x)\succeq 0$. Then for $\alpha\leq \frac{1}{2p}\eps^{(2-p)/p}$ we get $X(x-\alpha \nabla f(x))\succeq 0$.
\label{theorem::schatten_psd_constraint}
\end{restatable}
For completeness, its proof is done in the appendix.

In practice, the singular value relaxation bound is much better, since the step bound tends to be larger.

\subsection{Algorithm.}
The results above allow us to introduce a gradient descent method, summarized in the algorithm \ref{alg::sing_val_grad} (for singular values relaxation). In this algorithm an abstract procedure ChooseStep returns the appropriate step, which is less or equal to $\frac{\eps}{4(1+\eps^q)}$. The second procedure RoundSolution corresponds to the rounding method, described in the introduction.

\begin{algorithm}[h]
\caption{Gradient descent for singular values relaxation}
\begin{algorithmic}[1]
\State $X_0=X_{SDP}$
\State $K=\left\lbrace X\left|X=X\trans,\ X\succeq 0,\ W\leq \tr (AX) \leq SDP\right.\right\rbrace$
\For{$n=1$:max\_iter}
\State $\alpha = \textrm{ChooseStep}(X_n)$
\State $X_{n} = X_{n-1}-4\alpha \eps(1+\eps^q)\brackets{XX\trans+\eps I}^{-2}X$
\State $\brackets{X_n}_{ii}=1$
\If{$X_n\not\in K$}
\State break
\EndIf
\If{$\|4\alpha \eps(1+\eps^q)\brackets{X_{n}X_{n}\trans+\eps I}^{-2}X_{n}\|_F < \mathrm{tol}$}
\State break
\EndIf
\EndFor
\State $\bx,W^{**}=\textrm{RoundSolution}(A, X_{n})$
\State \Return $\bx, W^{**}$
\end{algorithmic}
\label{alg::sing_val_grad}
\end{algorithm}

\section{Experiments}
\subsection{Setup}
We have tested both CVX and SDPLR solvers followed by our algorithm on Gset graphs collection \url{https://web.stanford.edu/~yyye/yyye/Gset/} (originally introduced in \cite{Helmberg1997}). We also tested the Biq Mac library \url{http://biqmac.uni-klu.ac.at/biqmaclib.html} (namely on Beasley instances \cite{Beasley1998}). The latter is a collection of \binary problems, which are converted to \pmone ones as discussed in the appendix.

For each solver we first applied it to the problem. Then we computed the maximum cut based on the \num{1e5} rounding operation. The number of rounding operations was chosen to be completely sure that the solution provides better results compared to others. After that we chose $\eps=0.005$ (smaller values led to ill-conditioned gradients), $q=0.8$ for the singular values relaxation (this value is used by the authors of \cite{Li}), $p=0.1$ and $p=0.01$ for Schatten norms. The stopping criteria for the gradient descent were $100$ iterations and Frobenius norm of the gradient (less than \num{1e-5}). After that a new cut was obtained after \num{1e5} rounding operations. Rank tolerance was chosen to be \num{1e-4}. For Schatten norms the step size was at most $\eps$, which is larger than the theoretical value. Nevertheless, for such steps we have not observed any violations of the PSD constraint.

We tested all methods on the first $21$ Gset graphs. These graphs have $800$ nodes and are solvable with CVX. Another $10$ graphs of size 2000 were tested with SDPLR only.

We also tested all methods on the Beasley instances from the Biq Mac library. We chose relatively large \binary problems with 250 and 500 nodes. They were tested for CVX only.

\begin{table}
\input{gset_table_sdplr.txt}
\caption{SDPLR improvement on Gset graphs}
\label{tab::gset_sdplr}
\end{table}

\begin{table}
\input{gset_table_cvx.txt}
\caption{CVX improvement on Gset graphs}
\label{tab::gset_cvx}
\end{table}

\begin{table}
\input{bigmac_250_cvx.txt}
\caption{CVX improvement on the Biq Mac graph collection (250 nodes)}
\label{tab::biqmac_250}
\end{table}

\begin{table}
\input{bigmac_500_cvx.txt}
\caption{CVX improvement on the Biq Mac graph collection (500 nodes)}
\label{tab::biqmac_500}
\end{table}

\subsection{Discussion}
Results are shown in tables \ref{tab::gset_sdplr} for SDPLR and \ref{tab::gset_cvx} for CVX. 

Our approach outperforms both solvers in terms of the maximum cut on approximately half of the graphs (bold numbers). Moreover, for the first $21$ graphs, it outperforms them on almost the same set of graphs. Among tested rank relaxations the best performance was shown by Schatten norm relaxation with $p=0.1$. Probably, $p=0.01$ was a worse choice as it resulted in bigger influence of the smoothing parameter $\eps$ as it is included in the gradient in order of $\eps^{(p-2)/2}$.

The same result is observed for the \binary problems, but the singular values relaxation showed the best performance.

All methods performed well for rank reducing. However, some runs resulted in extremely large ranks. It show, that the resulting point has a lot of relatively small singular values, which are however not thresholded by \num{1e-4} (note that all singular values sum to the graph's size). This behavior might be considered a drawback of the Schatten norms relaxation and might be caused by the minimization of the sum of singular values (compared to the sum of fractions in the singular values relaxation). Note that small singular values have little effect in the rounding procedure, so this drawback does not have great influence on the method's performance. 

\section{Conclusion}
We developed an efficient first-order procedure which is aimed at improvement of SDP relaxation solutions for quadratic binary programming. We relax rank minimization in terms of the objective function and the linear constraint that controls optimality of the initial SDP objective. Rank function relaxation is performed by either singular value relaxation or Schatten p-norms. The latter approach showed the best performance on Gset graphs, while the singular values relaxation performed best on the Beasley instances.

\bibliographystyle{IEEEtran}
\bibliography{biblio}

\section{Appendix}
\subsection{Correspondence between \{0,1\} and \{-1, 1\} problems}
This part mostly copies derivations from the paper \cite{Helmberg1995} and also clarifies some of the steps. It is included for completeness.

Consider a following problem:
\[
\begin{split}
\max\ &x\trans Ax+b\trans x,\\
\st\ &x\in\{0, 1\}^n.
\end{split} 
\]
Since it is a $\{0,1\}$-problem, we substitute $B=A+\diag (b)$ and obtain an equivalent formulation:
\[
\begin{split}
\max\ &x\trans Bx,\\
\st\ &x\in\{0, 1\}^n.
\end{split}
\]
To get a $\{-1,1\}$-problem, we denote $y=2x-1$:
\[
\begin{split}
\max\ &\frac{1}{4}(y+1_n)\trans B(y+1_n),\\
\st\ &y\in\{-1, 1\}^n.
\end{split}
\]
Note, that:
\[
(y+1_n)\trans B(y+1_n)=y\trans By+2y\trans B1_n+1_n\trans B1_n.
\]
If we denote $B1_n=c\in \RR^n,\ 1_n\trans B1_n=d\in\RR$, this problem might be expressed as
\[
\begin{split}
\max\ &\frac{1}{4}z\trans Cz,\\
\st\ &z\in\{-1, 1\}^{n+1},\\
&z_{n+1}=1,
\end{split}
\]
where the matrix $C$ is a block matrix:
\[
C=
\begin{pmatrix}
B & c\\
c\trans & d
\end{pmatrix}.
\]
We obtained an equivalent formulation of the initial problem. The last constraint of the final problem, which is $z_{n+1}=1$, is not necessary. It does not affect on the objective since it is a quadratic problem, hence one might flip the values of $z$ as $z\rightarrow -z$ to get a solution of the initial $\{0,1\}$-problem in case of $z_{n+1}=-1$.

We are also interested in the case where $C\succeq 0$, which immediately provides guarantees for an SDP solution. Using Schur complement, we observe that
\[
\begin{pmatrix}
B & c\\
c\trans & d
\end{pmatrix} \succeq 0 \Longleftrightarrow
dB\succeq cc\trans.
\]
If eigenspaces of $B$ and $cc\trans$ are such that this condition might be satisfied for some $d$, we might change it (as it is actually a constant and does not affect maximization). However, an increased value of $d$ would provide a $2/\pi$ bound for a different problem with bigger objective value. Hence, a true $2/\pi$ bound for the initial problem must be obtained without changes in the constant part of the objective.


\subsection{Gradient descent step choice for Schatten norm relaxation}
\schattentheorem*
\begin{proof}
The gradient step in the upper-triangular parametrization is equivalent to the ordinary gradient step (multiplied by $2$) with substitution of diagonal elements with units. We are going to show that the first step results in a PSD matrix and the second step keeps matrix PSD.

Consider a symmetric PSD point and its SVD decomposition $X=USU\trans$. Hence for the step $\alpha$ the new point is
\[
\begin{split}
&X-2\alpha pX\brackets{X\trans X+\eps I}^{(p-2)/2}=\\
&=USU\trans-2\alpha pUSU\trans\brackets{US^2U\trans+\eps I}^{(p-2)/2}=\\
&=U\brackets{S-2\alpha pS\brackets{S^2+\eps I}^{(p-2)/2}}U\trans.
\end{split}
\]

Hence positive-semidefiniteness of the resulting matrix is equivalent to such characteristic of the expression in brackets. It is a diagonal matrix. 

If $S_{ii}=0$, then the corresponding diagonal elements are obviously zero. Otherwise we need them to be positive:
\[
\begin{split}
&S_{ii}-2\alpha pS_{ii}\brackets{S_{ii}^2+\eps I}^{(p-2)/2}\geq 0\Rightarrow\\
&\Rightarrow \alpha \leq \frac{\brackets{S_{ii}^2+\eps I}^{(2-p)/2}}{2 p}.
\end{split}
\]
Therefore, it suffices to take
\[
\alpha \leq \frac{\eps^{(2-p)/2}}{2p}.
\]

Now we want to show that substituting diagonal elements with units is equivalent to adding a diagonal matrix with non-negative entries. In this case a sum of two PSD matrix is PSD.

It is also equivalent to the fact that all diagonal elements of the gradient are non-negative. This is always true for a symmetric and PSD matrix $X=USU\trans$:
\[
\begin{split}
&\brackets{X\brackets{X\trans X+\eps I}^{(p-2)/2}}_{ii}=\brackets{US\brackets{S^2+\eps I}^{(p-2)/2}U\trans}_{ii}=\\
&=\sum_{j}U_{ij}U_{ij}S_{jj}\brackets{S_{jj}^2+\eps I}^{(p-2)/2}=\\
&=\sum_{j}U_{ij}^2S_{jj}\brackets{S_{jj}^2+\eps I}^{(p-2)/2}\geq 0.
\end{split}
\]
This observation completes the proof.
\end{proof}

\end{document}

%% file: gset_table_sdplr.txt
\begin{tabularx}{\columnwidth}{|X|XX|XX|XX|XX|}
\toprule
{}
 & \multicolumn{2}{l}{SDP}& \multicolumn{2}{l}{singular} & \multicolumn{2}{l}{Schatten} & \multicolumn{2}{l}{Schatten} \\
  & \multicolumn{2}{l}{ }& \multicolumn{2}{l}{values} & \multicolumn{2}{l}{p=0.1} & \multicolumn{2}{l}{p=0.01} \\
\hline
\# & rank & cut & rank & cut & rank & cut & rank & cut \\
\midrule
1  &        14 &       11466 &            13 &           11448 &            15 &           11451 &           13 &          11459 \\
\textbf{2}  &        15 &       11436 &            13 &           11438 &            13 &           \textbf{11456} &           14 &          11430 \\
\textbf{3}  &        14 &       11446 &            14 &           11445 &           461 &           \textbf{11455} &           26 &          11453 \\
\textbf{4}  &        14 &       11487 &            14 &           11475 &           319 &           \textbf{11511} &           14 &          11497 \\
\textbf{5}  &        13 &       11462 &            12 &           11462 &            18 &           11451 &           12 &          \textbf{11471} \\
6  &        13 &        2026 &            13 &            1989 &           105 &            2013 &           13 &           2012 \\
\textbf{7}  &        12 &        1833 &            12 &            1821 &            12 &            \textbf{1834} &           12 &           1822 \\
\textbf{8}  &        12 &        1834 &            11 &            1833 &            11 &            \textbf{1840} &           12 &           1831 \\
9  &        12 &        1879 &            12 &            1872 &            12 &            1869 &           12 &           1875 \\
10 &        12 &        1841 &            12 &            1829 &            12 &            1818 &           12 &           1820 \\
11 &        22 &         538 &           138 &             538 &           718 &             538 &          273 &            536 \\
\textbf{12} &        13 &         532 &            27 &             \textbf{536} &           190 &             534 &           41 &            534 \\
13 &        11 &         562 &             8 &             560 &            24 &             562 &            8 &            562 \\
\textbf{14} &        14 &        2994 &            13 &            2995 &            22 &            2992 &           14 &           \textbf{2999} \\
\textbf{15} &        16 &        2979 &            99 &            2982 &           173 &            \textbf{2986} &          122 &           2983 \\
\textbf{16} &        16 &        2978 &            14 &            2981 &           195 &            \textbf{2984} &           15 &           \textbf{2984} \\
17 &        16 &        2978 &            13 &            2975 &           152 &            2974 &           15 &           2974 \\
\textbf{18} &        11 &         924 &            10 &             \textbf{930} &            12 &             921 &           11 &            929 \\
\textbf{19} &        10 &         850 &             9 &             846 &             9 &             \textbf{854} &            9 &            851 \\
\textbf{20} &         9 &         888 &             8 &             882 &             8 &             \textbf{889} &            8 &            884 \\
21 &        10 &         868 &            33 &             864 &            52 &             862 &           41 &            863 \\
\textbf{22} &        18 &       13008 &            18 &           13006 &            18 &           13003 &           18 &          \textbf{13025} \\
23 &        20 &       13010 &            52 &           12985 &            78 &           13010 &           55 &          13004 \\
\textbf{24} &        19 &       13000 &            18 &           13004 &            22 &           13005 &           19 &          \textbf{13010} \\
\textbf{25} &        19 &       13006 &            19 &           12988 &           196 &           \textbf{13026} &           19 &          12987 \\
\textbf{26} &        19 &       12971 &           132 &           12985 &           181 &           12969 &          120 &          \textbf{12990} \\
\textbf{27} &        18 &        2988 &            17 &            2988 &            17 &            \textbf{3027} &           18 &           2989 \\
28 &        20 &        2956 &            20 &            2948 &            54 &            2947 &           20 &           2956 \\
\textbf{29} &        18 &        3044 &            17 &            \textbf{3056} &            23 &            3050 &           17 &           3038 \\
\textbf{30} &        17 &        3076 &            17 &            \textbf{3081} &            17 &            3076 &           17 &           3067 \\
\textbf{31} &        19 &        2947 &            18 &            2955 &            18 &            \textbf{2959} &           18 &           2958 \\
\bottomrule
\end{tabularx}

%% file: gset_table_cvx.txt
\begin{tabularx}{\columnwidth}{|X|XX|XX|XX|XX|}
\toprule
{}
 & \multicolumn{2}{l}{SDP}& \multicolumn{2}{l}{singular} & \multicolumn{2}{l}{Schatten} & \multicolumn{2}{l}{Schatten} \\
  & \multicolumn{2}{l}{ }& \multicolumn{2}{l}{values} & \multicolumn{2}{l}{p=0.1} & \multicolumn{2}{l}{p=0.01} \\
\hline
\# & rank & cut & rank & cut & rank & cut & rank & cut \\
\midrule
1  &        13 &       11462 &            13 &           11448 &            15 &           11451 &           13 &          11456 \\
\textbf{2}  &        13 &       11436 &            13 &           11438 &            13 &           \textbf{11456} &           13 &          11433 \\
\textbf{3}  &        14 &       11446 &            14 &           11445 &           502 &           11446 &           28 &          \textbf{11453} \\
\textbf{4}  &        14 &       11487 &            14 &           11471 &           304 &           \textbf{11511} &           14 &          11497 \\
\textbf{5}  &        12 &       11462 &            12 &           11464 &            18 &           11451 &           12 &          \textbf{11471} \\
6  &        13 &        2024 &            13 &            1994 &           108 &            2013 &           13 &           2016 \\
7  &        13 &        1833 &            12 &            1821 &            12 &            1828 &           12 &           1822 \\
\textbf{8}  &        12 &        1835 &            12 &            \textbf{1856} &           109 &            1846 &           13 &           1839 \\
9  &        12 &        1879 &            12 &            1872 &            12 &            1869 &           12 &           1875 \\
10 &        12 &        1841 &            12 &            1825 &            12 &            1820 &           12 &           1836 \\
11 &        10 &         534 &             6 &             534 &             6 &             534 &            7 &            534 \\
\textbf{12} &         9 &         532 &             8 &             534 &            29 &             \textbf{536} &            8 &            \textbf{536} \\
13 &         8 &         562 &             8 &             560 &            76 &             562 &            8 &            560 \\
\textbf{14} &        13 &        2994 &            13 &            2995 &            22 &            2994 &           13 &           \textbf{2999} \\
\textbf{15} &        13 &        2979 &            13 &            2982 &            51 &            \textbf{2987} &           13 &           2981 \\
\textbf{16} &        14 &        2982 &            14 &            2981 &           589 &            2979 &           61 &           \textbf{2986} \\
17 &        13 &        2978 &            13 &            2978 &           439 &            2973 &           24 &           2976 \\
\textbf{18} &        10 &         924 &            10 &             \textbf{930} &            11 &             920 &           10 &            929 \\
\textbf{19} &         9 &         847 &             9 &             846 &             9 &             \textbf{850} &            9 &            846 \\
\textbf{20} &         9 &         882 &             9 &             887 &           222 &             \textbf{888} &           20 &            886 \\
\textbf{21} &         9 &         862 &             9 &             865 &             9 &             \textbf{867} &            9 &            865 \\
\bottomrule
\end{tabularx}

%% file: bigmac_250_cvx.txt
\begin{tabularx}{\columnwidth}{|X|XX|XX|XX|XX|}
\toprule
{}
 & \multicolumn{2}{l}{SDP}& \multicolumn{2}{l}{singular} & \multicolumn{2}{l}{Schatten} & \multicolumn{2}{l}{Schatten} \\
  & \multicolumn{2}{l}{ }& \multicolumn{2}{l}{values} & \multicolumn{2}{l}{p=0.1} & \multicolumn{2}{l}{p=0.01} \\
\hline
\# & rank & cut & rank & cut & rank & cut & rank & cut \\
\midrule
1  &          6 &        45369 &              6 &            45369 &            60 &           45369 &            173 &            45369 \\
2  &          6 &        44579 &              6 &            44513 &             7 &           44571 &             47 &            44515 \\
3  &          6 &        48857 &              6 &            48857 &            13 &           48833 &             81 &            48857 \\
4  &          7 &        41094 &              7 &            \textbf{41116} &            17 &           41094 &             88 &            \textbf{41116} \\
5  &          5 &        47685 &              5 &            \textbf{47738} &            16 &           47679 &             75 &            47685 \\
6  &          7 &        40519 &              7 &            \textbf{40545} &             9 &           40475 &             70 &            40469 \\
7  &          6 &        46605 &              6 &            46563 &             6 &           46659 &             38 &            \textbf{46671} \\
8  &          7 &        35076 &              7 &            35000 &             8 &           35076 &             75 &            \textbf{35079} \\
9  &          6 &        48454 &              6 &            \textbf{48570} &             9 &           48447 &             59 &            48364 \\
10 &          6 &        39944 &              6 &            \textbf{39990} &            15 &           39974 &             77 &            39944 \\
\bottomrule
\end{tabularx}

%% file: bigmac_500_cvx.txt
\begin{tabularx}{\columnwidth}{|X|Xl|Xl|Xl|Xl|}
\toprule
{}
 & \multicolumn{2}{l}{SDP}& \multicolumn{2}{l}{singular} & \multicolumn{2}{l}{Schatten} & \multicolumn{2}{l}{Schatten} \\
  & \multicolumn{2}{l}{ }& \multicolumn{2}{l}{values} & \multicolumn{2}{l}{p=0.1} & \multicolumn{2}{l}{p=0.01} \\
\hline
\# & rank & cut & rank & cut & rank & cut & rank & cut \\
\midrule
1  &          9 &       114540 &              9 &           114440 &            37 &          113880 &            246 &           114100 \\
2  &          8 &       127280 &              8 &           127230 &             8 &          127080 &             16 &           \textbf{127370} \\
3  &          9 &       129080 &              9 &           129210 &            17 &          129130 &            148 &           \textbf{129400} \\
4  &          9 &       128000 &              9 &           128040 &            43 &          127940 &            273 &           \textbf{128170} \\
5  &          8 &       123570 &              8 &           \textbf{123860} &             8 &          123430 &             35 &           123480 \\
6  &          8 &       119770 &              8 &           \textbf{120350} &             8 &          119750 &              8 &           119590 \\
7  &          9 &       120210 &              9 &           119920 &            15 &          120100 &            177 &           120070 \\
8  &          9 &       121940 &              9 &           121650 &            12 &          \textbf{121980} &            101 &           121920 \\
9  &          9 &       118700 &              9 &           118570 &            26 &          118360 &            209 &           118420 \\
10 &          7 &       129220 &              7 &           129100 &             7 &          129120 &             26 &           129040 \\
\bottomrule
\end{tabularx}